\documentclass[10pt,reqno]{amsart}

\usepackage{color}
\usepackage{graphicx}
\usepackage{amsthm}

\newtheorem{thm}{Theorem}[section]

\newtheorem{lemma}[thm]{Lemma}

\theoremstyle{definition}
\newtheorem{defn}[thm]{Definition}
\newtheorem{example}[thm]{Example}
\newtheorem{assumption}[thm]{Assumption}
\theoremstyle{remark}
\newtheorem{remark}[thm]{Remark}

\def\qed{{\hfill $\Box$ \bigskip}}

\def\XXint#1#2#3{{\setbox0=\hbox{$#1{#2#3}{\int}$}
\vcenter{\hbox{$#2#3$}}\kern-.5\wd0}}

\def\<{\langle}
\def\>{\rangle}

\newcommand\bH{\mathbb{H}}
\newcommand\bZ{\mathbb{Z}}

\newcommand\fR{\mathbf{R}}

\newcommand\cA{\mathcal{A}}

\newcommand\cF{\mathcal{F}}
\newcommand\cG{\mathcal{G}}

\newcommand\cI{\mathcal{I}}

\newcommand\cR{\mathcal{R}}

\newcommand{\mysection}[1]{\section{#1}
\setcounter{equation}{0}}

\begin{document}

\title[Calder\'on-Zygmund Approach for pseudo-differential operators]{An $L_q(L_p)$-theory for parabolic pseudo-differential equations: Calder\'on-Zygmund approach}

\author{Ildoo Kim}
\address{Department of Mathematics, Korea University, 1 Anam-dong, Sungbuk-gu, Seoul,
136-701, Republic of Korea} \email{waldoo@korea.ac.kr}

\author{Kyeong-Hun Kim}
\address{Department of Mathematics, Korea University, 1 Anam-dong,
Sungbuk-gu, Seoul, 136-701, Republic of Korea}
\email{kyeonghun@korea.ac.kr}
\thanks{This work was supported by Samsung Science  and Technology Foundation under Project Number SSTF-BA1401-02}

\author{Sungbin Lim}
\address{Department of Mathematics, Korea University, 1 Anam-dong, Sungbuk-gu, Seoul,
136-701, Republic of Korea} \email{sungbin@korea.ac.kr}



\begin{abstract}
In this paper we present a Calder\'{o}n-Zygmund approach for a large class of parabolic equations with  pseudo-differential operators $\mathcal{A}(t)$ of arbitrary order $\gamma\in(0,\infty)$. It is assumed that $\cA(t)$ is  merely measurable with respect to the time variable.  The unique solvability of the equation
$$
\frac{\partial u}{\partial t}=\cA u-\lambda u+f, \quad (t,x)\in \fR^{d+1}
$$
 and  the $L_{q}(\fR,L_{p})$-estimate
$$
\|u_{t}\|_{L_{q}(\fR,L_{p})}+\|(-\Delta)^{\gamma/2}u\|_{L_{q}(\fR,L_{p})}
+\lambda\|u\|_{L_{q}(\fR,L_{p})}\leq N\|f\|_{L_{q}(\fR,L_{p})}
$$
are obtained for  any $\lambda > 0$  and $p,q\in (1,\infty)$.
\end{abstract}

\maketitle

\mysection{Introduction}

Calder\'{o}n-Zygmund theorem has been a powerful tool in the theory of both elliptic and parabolic differential equations. See, for instance,   \cite{caffarelli1998w, gilbarg2001elliptic, iwaniec1983projections} (elliptic equations) and
 \cite{fabes1966singular, jones1964class, krylov2001calderon, Krylov2002} (2nd order parabolic equations).
In particular, Krylov \cite{krylov2001calderon, Krylov2002} introduced a  Calder\'on-Zygmund approach to obtain  $L_{q}(\mathbf{R},L_{p})$ and $L_p(\fR, C^{2+\alpha})$ -estimates for the second-order parabolic equations with merely measurable  coefficients with respect to the time variable.

In this article we use a Calder\'on-Zygmund approach to study the parabolic equation
\begin{equation}
                  \label{eqn 8.13}
   \frac{\partial u}{\partial t}=\cA u-\lambda u+f, \quad (t,x)\in \fR^{d+1}.
   \end{equation}
It is assumed that the  pseudo-differential operator $\cA(t)$  is merely measurable in $t$ and its symbol $\psi(t,\xi)$
satisfies
\begin{equation}
                     \label{eqn sym1}
\Re[-\psi(t,\xi)]\geq  \kappa|\xi|^{\gamma},\quad \quad \forall \, \xi\in\mathbf{R}^{d}
\end{equation}
and
\begin{equation}
               \label{eqn sym2}
|D_{\xi}^{\alpha}\psi(t,\xi)|\leq\kappa^{-1}|\xi|^{\gamma-|\alpha|},\quad \quad \forall \, \xi\in\mathbf{R}^{d}\setminus\{0\}, \,\, |\alpha|\leq\lfloor\frac{d}{2}\rfloor+1
\end{equation}
for some $\gamma, \kappa>0$. No  regularity condition of $\cA(t)$ in the time variable is assumed and the differentiability condition of $\psi(t,\xi)$ with respect to $\xi$ is only up to order  $\lfloor\frac{d}{2}\rfloor+1$.  Conditions (\ref{eqn sym1}) and (\ref{eqn sym2}) are  satisfied by a large class of pseudo-differential operators including $2m$-order differential operators and integro-differential operators. See Section 2 for some examples.  We only mention that if $\cA_1(t)$ and $\cA_2(t)$ satisfy the conditions with $\gamma_1$ and $\gamma_2$ respectively then for any constants $a,b>0$ the operator $\cA_{a,b}=-(-\cA_1)^a(-\cA_2)^b$ satisfies the conditions with $\gamma=a\gamma_1+b\gamma_2$ if for instance the symbols of $\cA_i$ are real-valued.

Our approach aims to  prove the $L_{q}(\fR,L_{p})$-estimate
 \begin{equation}
              \label{lqlp}
\|u_{t}\|_{L_{q}(\fR,L_{p})}+\|(-\Delta)^{\gamma/2}u\|_{L_{q}(\fR,L_{p})}
+\lambda\|u\|_{L_{q}(\fR,L_{p})}\leq N\|f\|_{L_{q}(\fR,L_{p})}
\end{equation}
 for any  $u\in C^{\infty}_0(\fR^{d+1})$ and $f:=u_t-\cA u+\lambda u$. We remark that  the classical multiplier theorem is not applicable to derive estimates like (\ref{lqlp}) because $\cA(t)$ is only measurable in $t$.

 We first  prove
 $$
 \lambda\|u\|_{L_{q}(\fR,L_{p})}\leq N\|f\|_{L_{q}(\fR,L_{p})}
 $$
  based on the  representation formula of solutions  and a few direct calculations.  Next we introduce a kernel $K(t,x,s,y)$ so that for any $u\in C^{\infty}_0(\fR^{d+1})$ and $f:=u_t-\cA u+\lambda u$ we have
 \begin{equation}
            \label{eqn 8.14}
 (-\Delta)^{\gamma/2}u(t,x)=\int_{\fR^{d+1}}K(t,x,s,y)f(s,y)dsdy=:\cG f(t,x).
 \end{equation}
Then, we prove
\begin{align}\label{goal}
\left\|\mathcal{G}f\right\|_{L_{q}(\mathbf{R},L_{p})} \leq N \|f\|_{L_{q}(\mathbf{R},L_{p})}.
\end{align}
The major step to prove (\ref{goal}) is to construct $(\mathbb{Q}_{m},m\in\mathbb{Z})$, a filtration of partitions of $\fR^{d+1}$ (see Definition \ref{def:filtraion}), and  show  that for any $Q\in \bigcup_{m\in \bZ}\mathbb{Q}_m$ the
following  H\"{o}rmander condition (cf. \cite{grafakos2008classical,grafakos2009modern,Krylov2002}) holds:
\begin{align}\label{key:hormander}
\sup_{(s,y),(r,z)\in Q}\int_{\mathbf{R}^{d+1}\setminus Q^{*}}|K(t,x,s,y)-K(t,x,r,z)|dxdt < \infty,
\end{align}
where $Q^*$ is an appropriate dilation of $Q$. The H\"ormander condition and the Calder\'on-Zygmumd theorem easily yield (\ref{goal}).

 It is  well known  that for the elliptic operators H\"{o}rmander condition is fulfilled if the related  kernel $K(x,y)$ is a \emph{standard kernel} i.e. $K(x,y)$ defined on $\mathbf{R}^{2d}\setminus\{(x,x):x\in\mathbf{R}^{d}\}$ satisfies
\begin{align*}
|K(x,y)|\leq \frac{N}{|x-y|^{d}}
\end{align*}
and
\begin{align*}
|K(x,y)-K(z,y)|\leq N\frac{|x-z|^{\alpha}}{|x-y|^{d+\alpha}}
\end{align*}
whenever $|x-z|\leq \frac{1}{2}\max\{|x-y|,|z-y|\}$ and
\begin{align*}
|K(x,y)-K(x,z)|\leq N\frac{|y-z|^{\alpha}}{|x-y|^{d+\alpha}}
\end{align*}
whenever $|y-z|\leq \frac{1}{2}\max\{|x-y|,|x-z|\}$ (see \cite{grafakos2009modern} for details).

In this article we  study a parabolic version of this result and investigate a sufficient condition on kernel $K(t,x,s,y)$ so that \eqref{key:hormander} holds for any $Q\in \bigcup_{m\in \bZ}\mathbb{Q}_m$. The filtration of partition $(\mathbb{Q}_{m},m\in\mathbb{Z})$ is constructed only according to the order of the operator $\cA(t)$. It turns out that if the order of $\cA(t)$ is not rational  then constructing appropriate filtration of partitions by itself is a quite challenging work.

To the best of our knowledge,  only few studies have been made on Calder\'{o}n-Zygmund approach for non second-order parabolic equations.
If $p=q$, a result on integro-differential operators  of the type
\begin{equation*}
\mathcal{A}^{(\alpha)}f:=\int_{\mathbf{R}^{d}}\left(f(x+y)-f(x)-1_{|y|\leq 1}(y\cdot\nabla f(x))\right) \frac{m(t,y)}{|y|^{d+\alpha}}, \quad \alpha\in (0,2)
\end{equation*}
was introduced in \cite{Mikulevivcius1992} under certain assumptions on $m(t,y)$.
 The version of Calder\'{o}n-Zygmund decomposition of $\mathbf{R}^{d+1}$ introduced in  \cite{Mikulevivcius1992} uses  non-congruent rectangles to construct $\mathbb{Q}_n$ for each $n$ and the non-congruency of such rectangles depends also on the given function $u$. We believe that  the constants in the $L_p$-estimates of \cite{Mikulevivcius1992} are not controllable due to such non-congruency and the proof of  \cite{Mikulevivcius1992} is incomplete.   In this article we use congruent cubes to construct $\mathbb{Q}_n$ and our construction depends only on the order of the operator $\cA(t)$.  Our results certainly cover that of \cite{Mikulevivcius1992} (see Example \ref{ex 3}).

Below are some related $L_p$-estimates on non-local parabolic equations based on different approaches. Recently in \cite{Kim2014BMOpseudo} the authors proved a priori estimate (\ref{lqlp}) for the case $p=q$ and $\lambda=0$ using a BMO-$L^{\infty}$ type estimate. However this approach by itself is not enough to treat  the case $p\neq q$. Moreover the unique solvability of equation (\ref{eqn 8.13}) is not obtained in \cite{Kim2014BMOpseudo}.
In \cite{zhang2013lp},  \eqref{goal} is proved for the symbol of order $\gamma\in(0,2)$ which can be represented by the L\'{e}vy-Khintchine's formula
$$
\psi(\xi)=\int_{\mathbf{R}^{d}}\left(1+i(\xi\cdot y)1_{|y|\leq 1}-\exp\{i\xi\cdot y\}\right)\nu(dy),
$$
where  $\nu$ is a L\'{e}vy measure controlled from the below and the above by the L\'evy measures of  two  $\alpha$-stable processes. This result is based on a probabilistic method regarding L\'{e}vy processes which is legitimate only if the symbol $\psi(\xi)$ is independent of $t$ and its order is in $(0,2)$.  In this article we do not use any probabilistic method and no restriction on the order and time regularity  of $\cA(t)$ is assumed.

The article is organized as follows. Our main results are formulated in Section 2.
In Section 3, we illustrate the division-merger procedure to construct the filtration of partitions we need.
 The proofs of main theorems  and some auxiliary results are given in Sections 4, 5,  and 6.

We finish the introduction with  some notation used in this article. As usual $\fR^{d}$ stands for the Euclidean space of points
$x=(x^{1},...,x^{d})$,  $B_r(x) := \{ y\in \fR^d : |x-y| < r\}$  and
$B_r
 :=B_r(0)$.
 For  multi-indices $\alpha=(\alpha_{1},...,\alpha_{d})$,
$\alpha_{i}\in\{0,1,2,...\}$, $x \in \fR^d$, and  functions $u(x)$ we set
$$
 u_{x^{i}}=\frac{\partial u}{\partial x^{i}}=D_{i}u,\quad \quad
D^{\alpha}u=D_{1}^{\alpha_{1}}\cdot...\cdot D^{\alpha_{d}}_{d}u,
$$
$$
x^\alpha = (x^1)^{\alpha_1} (x^2)^{\alpha_2} \cdots (x^d)^{\alpha_d},\quad \quad
|\alpha|=\alpha_{1}+\cdots+\alpha_{d}.
$$
We also use $D^m_x$ to denote a partial derivative of order $m$  with respect to $x$.
For an open set $\Omega \subset \fR^d$
by $C^{\infty}_0(\Omega)$  we denote the set of infinitely differentiable  functions with compact support in $U$. For a Banach space $F$ and $p>1$ by $L_p(U,F)$ we denote the set of $F$-valued measurable functions $u$ on $\Omega$ satisfying
$$
\|u\|_{L_{p}(\Omega,F)}=\left(\int_{\Omega}\|u(x)\|_{F}^{p}dx\right)^{1/p}<\infty.
$$
We write $f\in L_{p,loc}(U,F)$ if  $\zeta f\in L_p(U,F)$ for any real-valued $\zeta\in C^{\infty}_0(U)$. Also $L_p(\Omega)=L_p(\Omega, \fR)$ and $L_p=L_p(\fR^d)$.
We use  ``$:=$" to denote a definition. $\lfloor a \rfloor$ is the biggest integer which is less than or equal to $a$.
By $\cF$ and $\cF^{-1}$ we denote the d-dimensional Fourier transform and the inverse Fourier transform, respectively. That is,
$\cF(f)(\xi) := \int_{\fR^{d}} e^{-i x \cdot \xi} f(x) dx$ and $\cF^{-1}(f)(x) := \frac{1}{(2\pi)^d}\int_{\fR^{d}} e^{ i\xi \cdot x} f(\xi) d\xi$.
For a Borel
set $A\subset \fR^d$, we use $|A|$ to denote its Lebesgue
measure and by $1_A(x)$ we denote  the indicator of $A$. $\textnormal{diam }A := \sup_{x,y \in A} |x-y|$. For a complex number $z$, $\Re[z]$ is the real part of $z$. Finally if we write $N=N(a,b,\ldots)$, this means that the constant $N$ depends only on $a,b,\ldots$.

\mysection{Main results}

Let $K(t,x,s,y)$ be a complex-valued measurable function on $\fR^{2d+2}$ satisfying $K(t,x,s,y)=K(t,x,s,y)1_{t>s}$. For $f \in C_0^\infty(\fR^{d+1})$ denote
$$
\cG f(t,x) = \int K(t,x,s,y)f(s,y)~dsdy.
$$
In this section we provide a sufficient condition on $K$ so that $\cG$ admits a weak type $(1,1)$ estimate, and using this result we obtain a $L_q(\fR, L_p)$ estimate for pseudo-differential operators $\cA(t)$.

Here is our assumption on the kernel $K$.
\begin{assumption}          \label{as 1}
There exist a constant $\gamma>0$ and  a nonnegative nondecreasing function $\varphi$ on $\fR_+$ such that

(i) for all $a>s$ and $y,z \in \fR^d$,
\begin{align}       \label{as 1 1}
\int_a^\infty \int_{\fR^d}|K(t,x,s,y)- K(t,x,s,z)|~dxdt \leq \varphi\big( \frac{|y-z|}{(a-s)^{1/\gamma}} \big);
\end{align}

(ii) for all $a>b\geq (s \vee r)$ and $y\in\mathbf{R}^{d}$
\begin{align}       \label{as 1 2}
\int_a^\infty \int_{\fR^d}| K(t,x,s,y)- K(t,x,r,y)|~dxdt \leq \varphi\big(\frac{|s-r|}{a-b} \big);
\end{align}

(iii) for all $b>s$ and $\rho>0$,
\begin{align}       \label{as 1 3}
\int_s^b \int_{|x-y| \geq \rho} |K(t,x,s,y)|~dxdt \leq \varphi\big( \frac{(b-s)^{1/\gamma}}{\rho} \big).
\end{align}

\end{assumption}

The proof of following results are given in Section \ref{pf main 11}.

\begin{thm}
                    \label{main 11}
Let $1<p \leq p_0$ and  Assumption \ref{as 1} hold. Assume that $\cG f$ is well defined for any $f\in C^{\infty}_0(\fR^{d+1})$ and the inequality
\begin{align}
                     \label{p0 conti}
\|\cG f \|_{L_{p_0}(\fR^{d+1})} \leq N_0 \|\cG f \|_{L_{p_0}(\fR^{d+1})}
\end{align}
 holds with some constant $N_0$  independent of $f$. Then the operator $\cG$ is uniquely extendable to a bounded operator on $L_p(\fR^{d+1})$
and satisfies the weak type $(1.1)$ estimate, (i.e.) for any $f \in C_0^\infty(\fR^{d+1})$ and $\alpha >0$
$$
\alpha \big| \{ (t,x) : \cG f(t,x) > \alpha \} \big| \leq N \|f\|_{L_1(\fR^{d+1})},
$$
where $N$ depends only on $d$, $p_0, \gamma, N_0$, and the function $\varphi$.
\end{thm}
\begin{thm}
            \label{main 22}
In addition to assumptions of Theorem \ref{main 11},
suppose $K(t,s,x,y)$ depends only on $(t,s,x-y)$,
and for all $t>s$ and $f \in C_0^\infty$
\begin{align}
                                    \label{pq con}
\Big\|\int_{\fR^d}K(t,x,s,y)f(y)dy \Big\|_{L_{p_0}} \leq \varphi(t-s) \|f\|_{L_{p_0}(\fR^d)}.
\end{align}
Then it holds that
$$
\|\cG f \|_{L_p(\fR,L_{p_0})}  \leq N \| f \|_{L_p(\fR,L_{p_0})}, \quad \forall f \in C_0^\infty(\fR, L_{p_0}),
$$
where $N$ depends only on $d$, $p, p_0, \gamma, N_0$, and the function $\varphi$.
\end{thm}

For any $p,q \in (1,\infty)$, by $\bH_{q,p}^{1,\gamma}=\bH_{q,p}^{1,\gamma}(\mathbf{R}^{d+1})$ we denote the space of distributions $u$ such that
$$
\|u\|_{L_{q}(\mathbf{R},L_{p})}<\infty,\quad \|u_{t}\|_{L_{q}(\mathbf{R},L_{p})}<\infty,\quad \|(-\Delta)^{\gamma/2}u\|_{L_{q}(\mathbf{R},L_{p})}<\infty.
$$
The norm of $u\in \bH_{q,p}^{1,\gamma}$ is defined by
$$
\|u\|_{\bH_{q,p}^{1,\gamma}}:=\|u\|_{L_{q}(\mathbf{R},L_{p})}+\|u_{t}\|_{L_{q}(\mathbf{R},L_{p})}+\|(-\Delta)^{\gamma/2}u\|_{L_{q}(\mathbf{R},L_{p})}.
$$
One can easily check that $\bH_{q,p}^{1,\gamma}$ is a Banach space.

Recall that the operator
$\mathcal{A}(t)$ has the  symbol $\psi(t,\xi)$, that is for $f \in C_0^\infty(\fR^{d+1})$
$$
\cF(\cA f) = \psi(t,\xi) \cF(u)(t,\xi).
$$
The following result is an application of Theorem \ref{main 22}. In the proof of Theorem \ref{main appl} we will take
$$
K(t,x,s,y)
=1_{s<t}\mathcal{F}^{-1} \left\{ |\xi|^{\gamma} \exp\left(\int_{s}^{t} \psi(r,\xi) dr\right) \right\}(x-y)
$$
so that (\ref{eqn 8.14}) holds for $\lambda=0$.

\begin{thm}
                    \label{main appl}
Let $p,q\in(1,\infty)$, $\lambda > 0$. Suppose that there exist constants $\gamma, \kappa>0$ so that
\begin{equation}
\Re[\psi(t,\xi)]\leq-\kappa|\xi|^{\gamma},\quad\xi\in\mathbf{R}^{d}\label{ellipticity}
\end{equation}
and for  all multi-index $\alpha$, $|\alpha|\leq\lfloor\frac{d}{2}\rfloor+1$,
\begin{equation}
|D_{\xi}^{\alpha}\psi(t,\xi)|\leq\kappa^{-1}|\xi|^{\gamma-|\alpha|},
\quad\xi\in\mathbf{R}^{d}\setminus\{0\}.\label{differentiability}
\end{equation}
 Then
for any $f\in L_{q}(\mathbf{R},L_{p})$, there exists a unique solution $u\in \bH_{q,p}^{1,\gamma}$ to equation \eqref{eqn 8.13}.
Furthermore, for this solution we have
\begin{align*}
\|u_{t}\|_{L_{q}(\mathbf{R},L_{p})}+\|(-\Delta)^{\gamma/2}u\|_{L_{q}(\mathbf{R},L_{p})}+\lambda\|u\|_{L_{q}(\mathbf{R},L_{p})}\leq \|f\|_{L_{q}(\mathbf{R},L_{p})},
\end{align*}
where $N=N(d,p,q,\kappa,\gamma)$.
\end{thm}

 Below we introduce  some examples related to  conditions (\ref{ellipticity}) and (\ref{differentiability}).

\begin{example}
                      \label{ex 1}
 The symbol of the $2m$-order operator
$$
A_1(t)u:=(-1)^{m-1}\sum_{|\alpha|=|\beta|=m}a^{\alpha \beta}(t)D^{\alpha+\beta}u,
$$
is $\psi(t,\xi)=-a^{\alpha \beta}(t) \xi^\alpha \xi^\beta$. Hence (\ref{ellipticity}) and (\ref{differentiability})  are satisfied if $a^{\alpha \beta}(t)$  are bounded complex-valued measurable functions satisfying
\begin{align*}
\kappa |\xi|^{2m} \leq  \sum_{|\alpha|=|\beta|=m}  \xi^\alpha \xi^\beta \Re\left[a^{\alpha \beta}(t)\right], \quad \forall \xi\in \fR^d.
\end{align*}

(ii) Similarly the $\gamma$-order nonlocal operator
$$
 \quad A_2(t):=-a(t)(-\Delta)^{\gamma/2}, \quad \quad \gamma\in (0,\infty)
$$
has symbol $\psi(t,\xi)=-a(t)|\xi|^{\gamma}$ and therefore for the above conditions it is sufficient to have
$$
\kappa<\Re[a(t)], \quad |a(t)|\leq \kappa^{-1}.
$$
\end{example}

The operator in Example \ref{ex 3} below is  considered in \cite{Mikulevivcius1992}.
\begin{example}
          \label{ex 3}
 Fix $\gamma\in (0,2)$ and  denote
$$
 \cA u:=\int_{\fR^d \setminus \{0\}} \Big(u(t,x+y)-u(t,x)-\chi(y)(\nabla u (t,x),y) \Big)
\frac{m(t,y) }{|y|^{d+\gamma}}dy
$$
where $\chi(y)= I_{\gamma >1} + I_{|y|\leq1}I_{\gamma=1}$. Then $\cA$ satisfies (\ref{ellipticity}) and (\ref{differentiability})  if $m(t,y)\geq 0$ is a  measurable function satisfying  the following  (see \cite{Kim2014BMOpseudo} for  details):

(1) If $\gamma=1$ then
$$
\int_{\partial B_1} w m(t,w)~S_1(dw)=0, \quad \forall t >0,
$$
where  $\partial B_{1}$ is the unit sphere in $\fR^d$ and $S_{1}(dw)$ is the surface measure on it.

  (2) The function $m=m(t,y) $ is zero-order homogeneous and $\lfloor\frac{d}{2}\rfloor+1$-times differentiable in $y$.

(3)   There is a constant $K$ such that for each  $t \in \fR$
$$
\sup_{|\alpha| \leq d_0, |y|=1} |D^\alpha_y m^{(\alpha)} (t,y) | \leq K.
$$

  (4) There  exists a constant $c>0$ so that $m(t,y)>c$ on a set  $E\subset \partial B_1$ of positive  $S_1(dw)$-measure.
\end{example}

Next we discuss the issue regarding the compositions and powers of operators.
Let $\cA_1(t)$ and $\cA_2(t)$ be linear operators with symbols $\psi_1(t)$ and  $\psi_2(t)$ satisfying the above prescribed conditions, that is there exist constants  $\gamma_1, \gamma_2, \kappa_1, \kappa_2>0$ so that
$$
      \Re [-\psi_i(t,\xi)]  \geq  \kappa_i |\xi|^{\gamma_i},\quad
 |D^\alpha \psi_i(t,\xi)| \leq \kappa_i^{-1} |\xi|^{\gamma_i -|\alpha|}, \quad (i=1,2),
 $$
 for any multi-index $\alpha$, $|\alpha|\leq \lfloor \frac{d}{2} \rfloor+1$. Fix $a,b>0$, and denote $\gamma:=a\gamma_1+b\gamma_2$. Consider $\gamma$-order  operator
 $$
 \cA_{a,b}(t)=-(-\cA_1(t))^a(-\cA_2(t))^b
 $$
 with the symbol $\psi=-(-\psi_1)^a(-\psi_2)^b$. It is easy to check that there exists a constant $N>0$ so that  for any multi-index $\alpha$, $|\alpha|\leq \lfloor \frac{d}{2} \rfloor+1$,
  $$
 |D^{\alpha} \psi (t,\xi)| \leq N |\psi|^{\gamma-\alpha}, \quad \quad  \xi \in \fR^d\setminus\{0\}.
$$
Therefore, (\ref{differentiability}) is satisfied, and Theorem \ref{main appl} is applicable to $\cA_{a,b}(t)$ if
 \begin{equation}
            \label{eqn 7.19}
  \Re [-\psi(t,\xi)]=\Re[(-\psi_1)^a(-\psi_2)^b]\geq N^{-1} |\xi|^{\gamma}, \quad \forall \, \xi\in \fR^d.
 \end{equation}
 Obviously (\ref{eqn 7.19}) is satisfied if, for instance, the symbols $\psi_i(t,\xi)$ are real-valued.

\mysection{Filtration of Partitions}

In this section we introduce a version of  Calder\'on-Zygmund theorem  we need. We also construct a filtration of partitions suitable for our pseudo-differential operators. Denote $\mathbb{N}=\{1,2,\cdots\}$ and $\mathbb{Z}=\{0, \pm 1, \pm 2, \cdots\}$.

\begin{defn}
                                         \label{def:filtraion}
Let $n\in \mathbb{N}$  and
$(\mathbb{Q}_{m},m\in\mathbb{Z})$ be a sequence of partitions of
$\mathbf{R}^{n}$ each consisting of disjoint bounded Borel subsets $Q\in\mathbb{Q}_{m}$.
We call $(\mathbb{Q}_{m},m\in\mathbb{Z})$ a \emph{filtration of partitions}
if
\begin{enumerate}
\item[(i)]  the partitions become finer as $m$ increases:
$$
\inf_{Q\in\mathbb{Q}_{m}}|Q|\rightarrow\infty\ \mbox{as}\ m\rightarrow-\infty,\quad\sup_{Q\in\mathbb{Q}_{m}}\textnormal{diam }|Q|\rightarrow0\ \mbox{as}\ m\rightarrow\infty ;
$$

\item[(ii)]  the partitions are nested: for each $m$ and $Q\in\mathbb{Q}_{m}$
there is a (unique) set $Q'\in\mathbb{Q}_{m-1}$ such that $Q\subset Q'$;

\item[(iii)]  the following regularity property holds: for $Q$ and $Q'$ as in
(ii) we have $|Q'|\leq N_{0}|Q|$, where $N_{0}$ is a constant independent
of $m,Q,Q'$.
\end{enumerate}
\end{defn}

\begin{example}
                              \label{ex2)parabolic}
For the second-order parabolic equations, $\mathbb{Q}_m$ on $\fR^{d+1}$ is typically defined by
$$
\mathbb{Q}_{m}=\{[i_{0}4^{-m},(i_{0}+1)4^{-m})\times Q_{m}(i_{1},\ldots,i_{d}),i_{0},i_{1},\ldots,i_{d}\in\mathbb{Z}\},
$$
where
$$
Q_m(i_1,\ldots, i_d):=[i_{1}2^{-m},(i_{1}+1)2^{-m})\times \cdots \times [i_{d}2^{-m},(i_{d}+1)2^{-m}).
$$
\end{example}

For Banach spaces $F$ and $G$,
$L(F,G)$ is the space of bounded linear operators from $F$ to $G$, and $L(F):=L(F,F)$. Define $B_{r}^{c}(x):=\{y\in\mathbf{R}^{n}:|x-y|\geq r\}$.
\begin{defn}
			\label{CZ kernel}
Let $(\mathbb{Q}_{m},m\in\mathbb{Z})$ be a filtration of partitions,
and for each $x,y\in\mathbf{R}^{n}$, $x\neq y$, let $K(x,y)$ be
a bounded operator from $F$ into $G$. We say that $K$ is a \emph{Calder\'on-Zygmund
kernel} relative to $(\mathbb{Q}_{m},m\in\mathbb{Z})$ if
\begin{enumerate}
\item[(i)]  there is a number $p_{0}\in(1,\infty)$ such that, for any $x$
and any $r>0$, $K(x,\cdot)\in L_{p_{0},loc}(B_{r}^{c}(x),L(F,G))$;
\item[(ii)]  for every $y\in\mathbf{R}^{n}$ the function $|K(x,y)-K(x,z)|$
is measurable as a function of $(x,z)$ on the set $\mathbf{R}^{2n}\cap\{(x,z):x\neq z,x\neq y\}$;
\item[(iii)]  there is a constant $N_{0}\geq1$ and, for each $Q\in\bigcup_{m\in\mathbb{Z}}\mathbb{Q}_{m}$,
there is a closed set $Q^{*}$ with the properties $\bar{Q}\subset Q^{*}$,
$|Q^{*}|\leq N_{0}|Q|$, and
\begin{equation}
\int_{\mathbf{R}^{n}\setminus Q^{*}}|K(x,y)-K(x,z)|dx\leq N_{0}\label{eq:KrylovHormanderCondition}
\end{equation}
whenever $y,z\in Q$.
\end{enumerate}
\end{defn}

 The following version of   the Calder\'on-Zygmund theorem is taken from \cite{krylov2001calderon}.

\begin{thm}
                                      \label{thm:main result1}
 Let $p>1$ and  $A:L_p(\mathbf{R}^{n},F)\rightarrow L_p(\mathbf{R}^{n},G)$ be a bounded linear operator. Assume that if $f\in C_{0}^{\infty}(\mathbf{R}^{n},F)$
then for almost any $x$ outside the support of $f$ we
have
$$
Af(x)=\int_{\mathbf{R}^{n}}K(x,y)f(y)dy
$$
where $K(x,y)$ is a Calder\'on-Zygmund kernel relative to a  filtration
of partitions. Then the operator $A$ is uniquely extendable to a
bounded operator from $L_{q}(\mathbf{R}^{n},F)$ to $L_{q}(\mathbf{R}^{n},G)$
for any $q\in(1,p]$, and $A$ is of weak type $(1,1)$ on smooth
functions with compact support.
\end{thm}

The filtration of partitions in Example \ref{ex2)parabolic} is not appropriate
for  pseudo-differential operators  since the kernels corresponding such operators  do not satisfy \eqref{eq:KrylovHormanderCondition} in the setting of Example \ref{ex2)parabolic}.
Finding an appropriate filtration of partitions requires delicate procedures unless the given order $\gamma$ is rational.
The remaining of this section  is devoted to construct a filtration of partitions for pseudo-differential operators of arbitrary order.

We fix $\gamma >0$ and   denote
$$
\mathbb{Q}_{0}^{(\gamma)}=\{Q_{0}\subset\mathbf{R}^{d+1}:Q_0=[i_{0},i_{0}+1)\times\prod_{j=1}^{d}[i_{j},i_{j}+1),\ i_{0},i_{1},\ldots,i_{d}\in\mathbb{Z}\}.
$$
To construct $\mathbb{Q}_{m}^{(\gamma)}$ we consider the cases $m\geq 0$ and $m<0$ separately.

First let $m=1,2,\cdots$. We construct  $\mathbb{Q}_{m}^{(\gamma)}$ inductively as follows. A similar  division procedure when $\gamma\in (0,2)$ can  be found in \cite{lara2014regularity}.
Suppose for a given $Q_{m-1} \in \mathbb{Q}_{m-1}^{(\gamma)}$, we can write
\begin{align}
				\label{eq:divisionprocedure}
Q_{m-1}=Q_{m-1}^{\textnormal{time}}\times Q_{m-1}^{\textnormal{space}}
\end{align}
where
$$
Q_{m-1}^{\textnormal{time}}=[i_{0}2^{-(m-1)\gamma}\tau_{m-1},(i_{0}+1)2^{-(m-1)\gamma}\tau_{m-1})
$$
$$
Q_{m-1}^{\textnormal{space}}=\prod_{j=1}^{d}[i_{j}2^{-(m-1)},(i_{j}+1)2^{-(m-1)}),
$$
for integers $i_0,i_1,\ldots,i_d$ and $\tau_{m-1} \in [1,2)$ (remember $\tau_{0}=1$).
Put
$$
2^{-(m-1)\gamma}\tau_{m-1}=2^{-m\gamma}\rho_m.
$$
Then
$$
\rho_m = 2^{\gamma} \tau_{m-1} \in [2^\gamma, \,2^{\gamma +1}) \subset  [2^{ \lfloor \gamma \rfloor}, \, 2^{ \lfloor \gamma \rfloor +2}).
$$
Put
$$
k_m=\begin{cases} \lfloor \gamma \rfloor &\text{if} \quad  \rho_m \in [2^{ \lfloor \gamma \rfloor}, \, 2^{ \lfloor \gamma \rfloor +1})\\
\lfloor \gamma \rfloor +1 &\text{if} \quad \rho_m \in [2^{ \lfloor \gamma \rfloor +1}, \, 2^{ \lfloor \gamma \rfloor +2}).
\end{cases}
$$
 We  split $Q_{m-1}^{\textnormal{space}}$
into $2^{d}$ congruent cubes and subdivide $Q_{m-1}^{\textnormal{time}}$
into $2^{k_m}$ congruent intervals. Taking all possible products
of subcubes and subintervals, we obtain the set of \emph{offsprings} of $Q_{m-1}$ (i.e.) $\{Q_{m}:Q_{m}\subset Q_{m-1}\}$ of the form
\begin{align}
				\label{m cube form}
Q_{m}=Q_{m}(i,l)=Q_{m}^{\textnormal{time}}(i)\times Q_{m}^{\textnormal{space}}(l)
\end{align}
where
$$
Q_{m}^{\textnormal{time}}(i)=[i_{0}2^{-m\gamma}\rho_m+(i-1)2^{-m\gamma}\frac{\rho_m}{2^{k_m}},\, i_{0}2^{-m\gamma}\rho_m +i2^{-m \gamma}\frac{\rho_m}{2^{k_m}})
$$
for some $1\leq i\leq2^{k_m}$ and
$$
Q_{m}^{\textnormal{space}}(l)=\prod_{j=1}^{d}[ i_{j}2^{-(m-1)}+(l_{j}-1)2^{-m} ,i_{j}2^{-(m-1)}+l_{j}2^{-m})
$$
for some $l=(l_{1},\ldots,l_{d})$, $l_{j}\in\{1,2\}$. Denoting
$$
\tau_m = \frac{\rho_m}{2^{k_m}}\in [1,2),
$$
 we can rewrite \eqref{m cube form} as
$$
Q_{m}=[\bar{i}_{0}2^{-m\gamma}\tau_m ,( \bar{i}_{0}+1) 2^{-m\gamma}\tau_m)\times \prod_{j=1}^{d}[ \bar{i}_{j}2^{-m} ,( \bar{i}_{j}+1)2^{-m}),
$$
where
$$
\bar{i}_{0}=2^{k_{m}}i_{0}+i-1,\quad \bar{i}_{j}=2i_{j}+l_{j}-1
$$
are integers for each $j=1,\ldots,d$.
Hence collecting all such  $Q_{m}\subset Q_{m-1}$ for every $Q_{m-1}\in\mathbb{Q}_{m-1}^{(\gamma)}$, we finally obtain the partition $\mathbb{Q}_{m}^{(\gamma)}$.  Moreover, since we choose $k_m$ such that
$\tau_m \in [1, 2)$,
by going back to \eqref{eq:divisionprocedure},
we can repeat the division procedure for $\mathbb{Q}_{m}^{(\gamma)}$ and generates $\mathbb{Q}_{m+1}^{(\gamma)}$.

Now we illustrate a merger procedure.
We define the collections of cubes $\mathbb{Q}_{m}^{(\gamma)}$ for $m=-1,-2,\ldots$ inductively.
Suppose that $\mathbb{Q}_{m+1}$ is a partition and every cubes $Q_{m+1}$ in $\mathbb{Q}_{m+1}^{(\gamma)}$ is of the form
\begin{align}
     \label{mer 1 as}
Q_{m+1}=Q_{m+1}^{\textnormal{time}}\times Q_{m+1}^{\textnormal{space}}
\end{align}
where
$$
Q_{m+1}^{\textnormal{time}}=[i_{0}2^{-(m+1)\gamma}\tau_{m+1},(i_{0}+1)2^{-(m+1)\gamma}\tau_{m+1}),
$$
$$
Q_{m+1}^{\textnormal{space}}=\prod_{j=1}^{d}[i_{j}2^{-(m+1)},(i_{j}+1)2^{-(m+1)}),
$$
$\tau_{m+1} \in [1,2)$ and $i_{0},i_1,\ldots,i_d$ are integers.
Obviously, $\mathbb{Q}^{(\gamma)}_0$ satisfies \eqref{mer 1 as} with $\tau_{0}=1$. We put
$$
2^{-(m+1)\gamma}\tau_{m+1}=2^{-m\gamma}\rho_m.
$$
Then
$$
\rho_m = 2^{-\gamma} \tau_{m+1} \in [2^{-\gamma} , 2^{-\gamma +1})\,\subset [2^{ -\lfloor \gamma \rfloor-1} , 2^{ -\lfloor \gamma \rfloor +1}).
$$
Put
$$
k_m=\begin{cases} \lfloor \gamma \rfloor &\text{if} \quad  \rho_m \in [2^{- \lfloor \gamma \rfloor}, \, 2^{ -\lfloor \gamma \rfloor +1})\\
\lfloor \gamma \rfloor +1 &\text{if} \quad \rho_m \in [2^{- \lfloor \gamma \rfloor -1}, \, 2^{ -\lfloor \gamma \rfloor}).
\end{cases}
$$
By combining cubes $Q_{m+1}$ of $\mathbb{Q}^{(\gamma)}_{m+1}$, we compose the partition $\mathbb{Q}_{m}^{(\gamma)}$ with $Q_m$ of the form
$$
Q_{m}=Q_{m}(i,l)=Q_{m}^{\textnormal{time}}(i)\times Q_{m}^{\textnormal{space}}(l)
$$
where
$$
Q_{m}^{\textnormal{time}}(i)=[i 2^{-m\gamma}\rho_m,(i+1) 2^{-m \gamma}\rho_m),
$$
for $i\in 2^{k_{m}}\mathbb{Z}$ and
$$
Q_{m}^{\textnormal{space}}(l)=\prod_{j=1}^{d}[l_{j}2^{-(m+1)},(l_{j}+1)2^{-(m+1)})
$$
for $l_{j}\in 2\mathbb{Z}$, $j=1,\ldots,d$.
Denote $\tau_m = 2^{k_m} \rho_m$. Then we can rewrite $Q_{m}(i,l)$ as
$$
Q_{m}=[\bar{i}_{0}2^{-m\gamma}\tau_m ,( \bar{i}_{0}+1) 2^{-m\gamma}\tau_m)\times \prod_{j=1}^{d}[ \bar{i}_{j}2^{-m} ,( \bar{i}_{j}+1)2^{-m}),
$$
where
$$
\bar{i}_{0}=\frac{i}{2^{k_{m}}},\quad \bar{i}_{j}=\frac{l_{j}}{2}
$$
are integers for each $j=1,\ldots,d$.
Furthermore, due to the choice of $k_m$ and $\rho_m$, we have
$$
\tau_m = 2^{k_m} \rho_m \in [1,2).
$$
Hence $\mathbb{Q}^{(\gamma)}_m$ satisfies \eqref{mer 1 as} with $m$ in place of $m+1$.
By repeating this merger procedure, we construct $\mathbb{Q}_{m}^{(\gamma)}$ for all $m=-1,-2,\ldots$.

\begin{remark}
                                       \label{Rmk : cubic is dyadic}
Due to the above procedure, one can
write $Q\in\mathbb{Q}_{m}^{(\gamma)}$ for $m\in\mathbb{Z}$ as follows
\begin{equation}
Q=Q(i_{0},\ldots,i_{d})=[i_{0}2^{-m\gamma}\tau_{m},(i_{0}+1)2^{-m\gamma}\tau_{m})\times\prod_{j=1}^{d}[i_{j}2^{-m},(i_{j}+1)2^{-m})\label{eq:cubic form}
\end{equation}
where $(i_{0},\ldots i_{d})\in\mathbb{Z}^{d+1}$ and $\tau_{m}\in[1,2)$.
In fact, $2^{-m\gamma}\tau_{m}$ is a dyadic number. Indeed, $\tau_{0}=1$
and recall that $2^{-m\gamma}\tau_{m}=2^{-k_{m}}2^{-(m-1)\gamma}\tau_{m-1}$
for $m=1,2,\ldots$. Therefore,
\begin{align*}
2^{-m\gamma}\tau_{m}=2^{-k_{m}}2^{-(m-1)\gamma}\tau_{m-1} & =2^{-k_{m}}2^{-k_{m-1}}2^{-(m-2)\gamma}\tau_{m-2}\\
 & =\ \cdots\\
 & =2^{-(k_{m}+k_{m-1}+\cdots+k_{2})}2^{-\gamma}\tau_{1} \\
 & =2^{-(k_{m}+k_{m-1}+\cdots+k_{1})}.
\end{align*}
Similarly, for $m=-1,-2,\ldots,$ we have $2^{-m\gamma}\tau_{m}=2^{k_{m}}2^{-(m+1)\gamma}\tau_{m+1}$
so that
$$
2^{-m\gamma}\tau_{m}=2^{k_{m}+k_{m+1}+\cdots+k_{-1}}.
$$
Therefore, $(\mathbb{Q}_{m}^{(\gamma)},m\in\mathbb{Z})$ is constituted of a class of dyadic cubes. If $\gamma=2$, then obviously
the above procedure generates the filtration of partitions in Example \ref{ex2)parabolic}.
\end{remark}

\begin{thm}
              \label{thm 8.19}
$(\mathbb{Q}_{m}^{(\gamma)},m\in\mathbb{Z})$ is a filtration of partitions.
\end{thm}

\begin{proof}
Due to the above division-merger procedures, (i) and (ii) of Definition
\ref{def:filtraion} are obvious. Hence it suffices to show the regularity
condition (iii). For $m\in\mathbb{Z}$, take $Q$ and $Q'$ such that
$Q\subset Q'$, $Q'\in\mathbb{Q}_{m}^{(\gamma)}$, and $Q\in\mathbb{Q}_{m+1}^{(\gamma)}$.
From \eqref{eq:cubic form}, we can write
$$
Q'(t',x')=[t',t'+2^{-m\gamma}\tau_{m})\times\prod_{j=1}^{d}[x'_{j},x'_{j}+2^{-m}),
$$
and
$$
Q(t,x)=[t,t+2^{-(m+1)\gamma}\tau_{m+1})\times\prod_{j=1}^{d}[x_{j},x_{j}+2^{-m-1}).
$$
Then by Remark \ref{Rmk : cubic is dyadic},
$$
\frac{|Q'(t',x')|}{|Q(t,x)|}=\frac{2^{-m\gamma-md}\tau_{m}}{2^{-(m+1)\gamma-(m+1)d}\tau_{m+1}}=2^{d+k_{m+1}}\leq2^{d+\lfloor\gamma\rfloor+1}.
$$
The theorem is proved.
\end{proof}

\mysection{Proof of Theorem \ref{main 11} and \ref{main 22}}
                                            \label{pf main 11}

We first  check the H\"ormander condition under  Assumption \ref{as 1}.

\begin{lemma}
                                  \label{main result2}
 Under Assumption \ref{as 1}, the
kernel $K(t,x,s,y)$ satisfies H\"ormander condition \eqref{eq:KrylovHormanderCondition}
with respect to $(\mathbb{Q}_{m}^{(\gamma)},m\in\mathbb{Z})$, the  filtration of partitions in Theorem \ref{thm 8.19}.
\end{lemma}

\begin{proof}
Let
$$
Q=[t_0, t_0+ 2^{-m\gamma} \tau_m) \times \prod_{j=1}^d [0, 2^{-m}), \quad m \in \bZ
$$
and
$$
Q^\ast=[t_0, t_0 + 4 \cdot 2^{-m\gamma}) \times \prod_{j=1}^d [-2 \cdot 2^{-m}, 2 \cdot 2^{-m}), \quad m \in \bZ,
$$
where $1 \leq \tau_m \leq 2$.

It suffices to show
\begin{equation}        \label{42 con}
\sup_{(s,y),(r,z) \in Q}\int_{\mathbf{R}^{d+1}\setminus Q^{*}}|K(t,x,s,y)-K(t,x,r,z)|~dxdt <\infty.
\end{equation}
Put
$$
\Gamma_1 =  \{ t \geq t_0+4 \cdot 2^{-m\gamma} \} \times \fR^d,
$$
and
$$
\Gamma_2 = \{ t_0 <t < t_0+4 \cdot 2^{-m\gamma} \} \cap (\fR^{d+1} \setminus Q^\ast).
$$
Recall that $K(t,x,s,y)$ vanishes if $t \leq s$.
Then obviously for any $(s,y),(r,z) \in Q$,
\begin{align*}
& \int_{\mathbf{R}^{d+1}\setminus Q^{*}}| K(t,x,s,y)- K(t,x,r,z)|~dxdt \\
&\leq \int_{\Gamma_1}| K(t,x,s,y)- K(t,x,r,z)|~dxdt
 +\int_{\Gamma_2}| K(t,x,s,y)- K(t,x,r,z)|~dxdt \\
&\leq \int_{\Gamma_1}|K(t,x,s,y)- K(t,x,s,z)|~dxdt
 + \int_{\Gamma_1}| K(t,x,s,z)- K(t,x,r,z)|~dxdt \\
&\quad + 2\sup_{(s,y) \in Q}\int_{\Gamma_2}| K(t,x,s,y)|~dxdt =: \cI_1 + \cI_2 + \cI_3.
\end{align*}

First we estimate $\cI_1$.	
Observe that
$$
t_0+ 4 \cdot 2^{-m\gamma} - s \geq 2^{-m \gamma},
$$
and $|z-y| \leq 2 \cdot 2^{-m}$.
So by \eqref{as 1 1},
\begin{align*}
\cI_1
&\leq \int_{t_0+4 \cdot 2^{-m\gamma}}^\infty \int_{\fR^d}  |K(t,x,s,y)- K(t,x,s,z)|~dxdt \\
&\leq \varphi\big( (t_0+4 \cdot 2^{-m\gamma}-s)^{-1/\gamma}|z-y| \big) \leq \varphi(2).
\end{align*}
For $\cI_2$ we use \eqref{as 1 2}.
Since
$ s,r \in [t_0, t_0 + 2 \cdot 2^{-m\gamma}]$,
we have
\begin{align*}
\cI_2
&\leq \int_{t_0+4 \cdot 2^{-m\gamma}}^\infty \int_{\fR^d}  K(t,x,s,z) -  K(t,x,r,z)|~dxdt \\
&\leq \varphi\big( (2 \cdot 2^{-m\gamma})^{-1}|s-r| \big) \leq \varphi( 1).
\end{align*}
Finally we estimate $\cI_3$. Note that for $(t,x) \in \Gamma_2$ and $(s,y) \in Q$
$$
|x-y| \geq 2^{-m}.
$$
Hence from \eqref{as 1 3},
\begin{align*}
\cI_3
&\leq 2\sup_{(s,y) \in Q} \int_{t_0}^{t_0+4 \cdot 2^{-m\gamma}} \int_{|x-y| \geq  2^{-m}}  K(t,x,s,y)~dxdt \\
&\leq 2\sup_{(s,y) \in Q} \int_{s}^{t_0+4 \cdot 2^{-m\gamma}} \int_{|x-y| \geq  2^{-m}}  K(t,x,s,y)~dxdt \\
&\leq 2\varphi\big( (4 \cdot 2^{-m\gamma})^{1/\gamma} ( 2^{-m})^{-1} \big) \leq 2\varphi( 4^{1/\gamma}),
\end{align*}
where the second inequality is because $K$ vanishes if $t \leq s$.
Therefore \eqref{42 con} is proved.
\end{proof}

We define the operator $\mathcal{K}(t,s)$ as follows :
\begin{align*}
\mathcal{K}(t,s)f(x)=\int_{\mathbf{R}^{d}}K(t,x,s,y)f(y)dy, \quad f \in C_0^\infty.
\end{align*}
Suppose that \eqref{pq con} holds, that is  for any $t>s$ and $f \in C_0^\infty$
\begin{align}
                \label{lp est}
\Big\|\int_{\fR^d}K(t,x,s,y)f(y)dy \Big\|_{L_{p_0}} \leq \varphi(t-s) \|f\|_{L_{p_0}}.
\end{align}
Since $\mathcal{K}(t,s)$ is linear and \eqref{lp est} holds, the operator $\mathcal{K}(t,x)$ is uniquely extendible to $L_{p_0}$.
Hence we can consider $\mathcal{K}(t,s)$ as a bounded operator on $L_{p_0}$. Denote
$$
\mathbb{Q}^{time}_m:=\{[4^{-m}i, 4^{-m}(i+1)) :  i\in \bZ\}, \quad m\in \bZ.
$$
\begin{lemma}
			\label{main result3}
Suppose that Assumption \ref{as 1} (ii) and \eqref{pq con} hold, and $K(t,s,x,y)=K(t,s,x-y)$. Then
$\mathcal{K}(t,s)$ satisfies the H\"ormander condition \eqref{eq:KrylovHormanderCondition} with $n=1$ and  	 $(\mathbb{Q}^{time}_{m},m\in\mathbb{Z})$
\end{lemma}

\begin{proof}

Let
$$
Q=[t_{0},t_{0}+\delta),\quad Q^{*}=[t_{0}-2\delta,t_{0}+2\delta).
$$
Note that for $t\notin Q^{*}$ and $s,r\in Q$, we have
$$
|s-r|\leq \delta,\quad |t-(t_{0}+\delta)|\geq \delta,
$$
and recall $K(t,s,x-y)=0$ if $t\leq s$.
Note
\begin{align}
                \label{norm of K difference}
\|\mathcal{K}(t,s)-\mathcal{K}(t,r)\|_{L(L_{p_0})}&=\sup_{\|f\|_{L_{p_0}}=1}\left\|\int_{\mathbf{R}^{d}}(K(t,s,x-y)-K(t,r,x-y))f(y)dy\right\|_{L_{p_0}}\nonumber
\\ &\leq \sup_{\|f\|_{L_{p_0}}=1}\|f\|_{L_{p_0}}\int_{\mathbf{R}^{d}}|K(t,s,x)-K(t,r,x)|dx\nonumber
\\ &= \int_{\mathbf{R}^{d}}|K(t,s,x)-K(t,r,x)|dx.
\end{align}

Therefore, by Assumption \ref{as 1} (ii) and \eqref{norm of K difference},
\begin{align*}
\int_{\mathbf{R}\setminus Q^{*}}\|\mathcal{K}(t,s)-\mathcal{K}(t,r)\|_{L(L_{p_0})} \,dt &\leq \int_{\mathbf{R}\setminus Q^{*}}\int_{\mathbf{R}^{d}}|K(t,s,x)-K(t,r,x)|dxdt
\\ &\leq \int_{t\geq t_{0}+2\delta}\int_{\mathbf{R}^{d}}|K(t,s,x)-K(t,r,x)|dxdt
\\ &\leq N\varphi(\frac{|s-r|}{\delta})\leq N\varphi(1) \leq N.
\end{align*}
The lemma is proved.
\end{proof}

\begin{proof}[{\bf{Proof of Theorem \ref{main 11} and \ref{main 22}}}]

Due to Lemma \ref{main result2} and Lemma \ref{main result3},
these are  easy consequences of Theorem \ref{thm:main result1}. We only mention that
in the proof of Theorem \ref{main 22}, following the proof of Theorem 1.1 of \cite{krylov2001calderon}, one can easily check that
for almost any $x$ outside of the closed support of $f \in C_0^\infty(\fR, L_{p_0}) $,
\begin{align*}
\cG f(t,x) =\int_{-\infty}^\infty\mathcal{K}(t,s)f(s,x)~ds,
\end{align*}
where $\cG$ denote the unique extension on $L_{p_0} (\fR^{d+1})$ stated in Theorem \ref{main 11}. The theorems are proved.
\end{proof}

\mysection{Auxiliary results}

In this section we study a kernel $p_{\lambda}$ and an operator $\cR_{\lambda}$ which are related to  $\cA(t)-\lambda$.

\begin{lemma}
				\label{frac est}
Let $\sigma\geq 0$,  $\delta>\sigma-\frac{d}{2}$ and  $h\in C^{1+\lfloor\sigma\rfloor}(\mathbf{R}^{d}\setminus\{0\})$. Suppose that
 there exists constant $c>0$ such that
\begin{equation}
                     \label{eqn 8.14.2}
|D^m_x(-\Delta)^{\lfloor\frac{\sigma}{2}\rfloor}h(x)|\leq c|x|^{\delta-2\lfloor\frac{\sigma}{2}\rfloor-m}\exp\{-c|x|^{\gamma}\},\quad\forall x\in\mathbf{R}^{d}\setminus\{0\}
\end{equation}
 for $m=0,1$. Also assume (\ref{eqn 8.14.2})  holds for $m=2$ if  $1 \leq \sigma-2\lfloor\frac{\sigma}{2}\rfloor<2 $.
Then
$$
\|(-\Delta)^{\sigma/2}h\|_{L_{2}(\mathbf{R}^{d})}\leq N
$$
where $N=N(c,d,\gamma,\delta,\sigma)$.\end{lemma}

\begin{proof}
 The case $\sigma\in [0,2)$ is proved in  \cite[Lemma 5.1]{Kim2014BMOpseudo}. For $\sigma \geq 2$, denote $$\tilde{\sigma}:=\sigma-2\lfloor\frac{\sigma}{2}\rfloor, \quad \quad
v:=(-\Delta)^{\lfloor\frac{\sigma}{2}\rfloor}h\in C^{1+\lfloor\sigma\rfloor-2\lfloor\frac{\sigma}{2}\rfloor}.
$$
Then
$$
\|(-\Delta)^{\sigma/2}h\|_{L_{2}(\mathbf{R}^{d})}=
\|(-\Delta)^{\frac{\sigma}{2}-\lfloor\frac{\sigma}{2}\rfloor}(-\Delta)^{\lfloor\frac{\sigma}{2}\rfloor}h\|_{L_{2}(\mathbf{R}^{d})}
=\|(-\Delta)^{\tilde{\sigma}/2}v\|_{L_{2}(\mathbf{R}^{d})}.
$$
Thus it is enough to  apply the result for $\sigma\in[0,2)$. The lemma is proved.
\end{proof}

Recall that $\psi(t,\xi)$ is the symbol of $\cA(t)$ satisfying
\begin{equation}
                  \label{14.3}
\Re[\psi(t,\xi)]\leq-\kappa|\xi|^{\gamma},
\end{equation}
$$
|D_{\xi}^{\alpha}\psi(t,\xi)|\leq\kappa^{-1}|\xi|^{\gamma-|\alpha|},\quad \quad  |\alpha|\leq\lfloor\frac{d}{2}\rfloor+1.
$$

Note that due to (\ref{14.3}),
$$
p_{\lambda}(t,s,x):=1_{s<t}\mathcal{F}^{-1}\left\{\exp\left(\int_{s}^{t}(\psi(r,\xi)-\lambda)dr\right)\right\}
$$
is well defined for any $\lambda\geq 0$. Similarly one can  check that
\begin{align*}
\mathcal{R}_{\lambda}f(t,x):=\mathcal{F}^{-1}\left(\int_{-\infty}^{\infty} 1_{s<t}
\exp \big(\int_{s}^{t}(\psi(r,\xi)-\lambda)dr \big) \cF f(s,\xi)  ~ds \right)(x)
\end{align*}
is well defined for  any  $\lambda > 0$ and $f \in L_2(\mathbf{R}^{d+1})$.
Obviously
\begin{align*}
\mathcal{R}_{\lambda}f(t,x)=\int_{-\infty}^{\infty}\int_{\mathbf{R}^{d}}p_{\lambda}(t,s,x-y)f(s,y)dyds,\quad \forall f\in C_0^\infty(\fR^{d+1}).
\end{align*}

In the following lemma we show that the operator $\mathcal{R}_{\lambda}$ is continuously extensible to $ L_{q}(\mathbf{R},L_{p})$ for any $p,q>1$.
\begin{lemma}
			\label{thm:norm of R}
Let  $\lambda >0$ and $p,q >1$.  Then
\begin{equation*}
     \|p_{\lambda}(t,s,\cdot)\|_{L_{1}}\leq N 1_{s<t}e^{-\lambda(t-s)},
     \end{equation*}
\begin{equation*}
        \|\mathcal{R}_{\lambda}f(t,\cdot)\|_{L_p} \leq N \lambda^{-(p-1)/p}\|f\|_{L_{p}(\mathbf{R}^{d+1})}, \quad \forall f \in C_0^\infty(\fR^{d+1}),
\end{equation*}
and
\begin{equation}
            \label{norm of R}
\|\mathcal{R}_{\lambda}f\|_{L_{q}(\mathbf{R},L_{p})}\leq \frac{N}{\lambda}\|f\|_{L_{q}(\mathbf{R},L_{p})}, \quad \forall f \in C_0^\infty(\fR^{d+1}),
\end{equation}
where $N=N(d,p,q,\kappa,\gamma)$.
\end{lemma}

\begin{proof}

(i) Using $xe^{-x}\leq 1$ for $x\geq 0$, one can check that  for any  multi-index $|\alpha|\leq\lfloor\frac{d}{2}\rfloor+1$
\begin{align*}
\Big|D^{\alpha}_{\xi}\exp\left(\int_{s}^{t}\psi(r,(t-s)^{-\frac{1}{\gamma}}\xi)dr \right) \Big| \leq N(\kappa)|\xi|^{\gamma - |\alpha|}e^{-\kappa |\xi|^{\gamma}}.
\end{align*}
Denote
\begin{align*}
q(t,s,x)&:=1_{s<t}\mathcal{F}^{-1}\left\{\exp\left(\int_{s}^{t}(\psi(r,(t-s)^{-\frac{1}{\gamma}}\xi))dr\right)\right\}
\\ &=(t-s)^{\frac{d}{\gamma}}e^{\lambda(t-s)}p_{\lambda}(t,s,(t-s)^{\frac{1}{\gamma}}x).
\end{align*}

By using H\"{o}lder inequality and Parseval's identity, for $\varepsilon< (\gamma/4\wedge 1/4)$,
\begin{align*}
&\|p_{\lambda}(t,s,\cdot)\|_{L_{1}}=\|(1+|x|^{\frac{d+\varepsilon}{2}})^{-1}(1+|x|^{\frac{d+\varepsilon}{2}})p_{\lambda}(t,s,\cdot) \|_{L_{1}}\\
&\leq N 1_{s<t} e^{-\lambda(t-s)}\left(\int_{\mathbf{R}^{d}}|(1+|x|^{\frac{d+\varepsilon}{2}})q(t,s,x)|^{2}dx\right)^{1/2}
\\ &\leq N1_{s<t} e^{-\lambda(t-s)} \|(1+(-\Delta)^{\frac{d+\varepsilon}{4}})\mathcal{F}(q)(t,s,\cdot)\|_{L_{2}}
\\ &\leq N 1_{s<t} e^{-\lambda(t-s)}\left(\|\mathcal{F}(q)\|_{L_{2}}
+\|(-\Delta)^{\frac{d+\varepsilon}{4}}\mathcal{F}(q)(t,s,\cdot)\|_{L_{2}}\right).
\end{align*}
Then by using Lemma \ref{frac est} with $\sigma = (d+\varepsilon)/2$ and $h(\cdot)=\mathcal{F}(q)(t,s,\cdot)$, we have
$$
\|p_{\lambda}(t,s,\cdot)\|_{L_{1}}\leq N(d,\gamma, \kappa) 1_{s<t}e^{-\lambda(t-s)}.
$$
To apply Lemma \ref{frac est}, $h$ should be $m+2\lfloor \frac{\sigma}{2}\rfloor$-times differentiable, and this is possible since $m+2\lfloor \frac{\sigma}{2}\rfloor\leq \lfloor\frac{d}{2}\rfloor+1$ for $m=0,1$ and $m=2$ if $1\leq \sigma-2\lfloor \frac{\sigma}{2}\rfloor <2$.

(ii) By Minkowski's inequality, Young's inequality, and (i),
\begin{eqnarray}
 \nonumber
\|\mathcal{R}_{\lambda}f(t,\cdot)\|_{L_p}&\leq& \int_{-\infty}^{\infty}1_{s<t}\|p_{\lambda}(s,t,\cdot)\|_{L_{1}}\|f(s,\cdot)\|_{L_{p}}ds \\
&=& N\int_{-\infty}^{\infty} 1_{s<t} e^{-\lambda(t-s)} \| f(s,\cdot)\|_{L_p} ds              \label{eqn 14.5} \\
&\leq& N \lambda^{-(p-1)/p}\|f\|_{L_{p}(\mathbf{R}^{d+1})}. \nonumber
\end{eqnarray}

(iii) By (\ref{eqn 14.5}),
\begin{align*}
\|\mathcal{R}f \|_{L_q(\fR,L_p)}
\leq \left\| \int_0^{\infty} e^{-\lambda s} \|f(t-s,\cdot) \|_{L_p} ds \right\|_{L_q(\fR)} \leq \frac{1}{\lambda} \|f \|_{L_q(\fR,L_p)}.
\end{align*}
The lemma is proved.
\end{proof}

\begin{remark}
                \label{cr ext}
Due to the above lemma, we can consider the continuous extension of  $\cR_{\lambda}$ on $L_{q}(\mathbf{R},L_{p})$ for any $p,q>1$.
From now on, we regard the operator $\cR_{\lambda}$ as this extension on $L_{q}(\mathbf{R},L_{p})$.
Actually $\cR_{\lambda}$ was already defined on $L_2(\fR^{d+1})$, but two different definitions coincide on  $L_2(\fR^{d+1}) \cap L_{q}(\mathbf{R},L_{p})$ due to  Riesz-Fischer theorem.
\end{remark}

\mysection{Proof of Theorem \ref{main appl}}
                                            \label{pf main appl}

                                            Define  $K(t,s,x,y)=K(t,s,x-y)$ by
\begin{equation}
              \label{eqn 14.6}
K(t,s,x)=1_{s<t}\mathcal{F}^{-1} \left\{ |\xi|^{\gamma} \exp\left(\int_{s}^{t} \psi(r,\xi) dr\right) \right\}.
\end{equation}
Also   define
\begin{align*}
\cG f(t,x)=\mathcal{F}^{-1}\left\{\int_{-\infty}^{\infty} 1_{s<t}  |\xi|^{\gamma} \exp\left(\int_{s}^{t} \psi(r,\xi) dr\right) \cF f(s,\xi)  ds \right\}.
\end{align*}
It is easy to check that  $\cG f$ is well defined if $f \in C_0^\infty (\fR^{d+1})$, and furthermore
\begin{align*}
\cG f(t,x)=\int_{-\infty}^{\infty}\int_{\mathbf{R}^{d}}K(t,s,x-y)f(s,y)dyds.
\end{align*}

\begin{thm}
Let $p,q\in(1,\infty)$.  Under the assumptions in Theorem \ref{main appl}, the kernel $K$ satisfies Assumption \ref{as 1} and (\ref{pq con}), and it holds that
\begin{align}\label{Gf estimate}
\|\mathcal{G}f\|_{L_{q}(\mathbf{R},L_{p})}\leq N \|f\|_{L_{q}(\mathbf{R},L_{p})}, \quad \forall\,\,f\in C^{\infty}_0(\fR^{d+1}),
\end{align}
where $N=N(d,p,q,\gamma,\kappa)$.
\end{thm}

\begin{proof}

{\bf{Part 1}}. We show that the  kernel $K$ defined in (\ref{eqn 14.6}) satisfies Assumption \ref{as 1}, \eqref{p0 conti}, and \eqref{pq con}.
Observe that
\begin{align*}
(t-s)^{1+\frac{d}{\gamma}} & K(t,s,(t-s)^{\frac{1}{\gamma}}x)\\
 & =N1_{s<t}(t-s)^{1+\frac{d}{\gamma}}\int_{\mathbf{R}^{d}}e^{i((t-s)^{\frac{1}{\gamma}}x,\xi)}|\xi|^{\gamma}\exp\left\{ \int_{s}^{t}\psi(r,\xi)dr\right\} d\xi\\
 & =N1_{s<t}(t-s)^{1+\frac{d}{\gamma}}\int_{\mathbf{R}^{d}}e^{i(x,(t-s)^{\frac{1}{\gamma}}\xi)}|\xi|^{\gamma}\exp\left\{ \int_{s}^{t} \psi(r,\xi)dr\right\} d\xi\\
 & =N1_{s<t}\int e^{i(x,\xi)}|\xi|^{\gamma}\exp\left\{\int_{s}^{t} \psi(r,(t-s)^{-\frac{1}{\gamma}}\xi)dr\right\} d\xi.
\end{align*}
Denote
$$F(t,s,\xi)=1_{s<t}|\xi|^{\gamma}\exp\left\{ \int_{s}^{t} \psi(r,(t-s)^{-\frac{1}{\gamma}}\xi)dr\right\}.
$$
 Due to  the assumptions on $\psi(t,\xi)$,
$$
|D_{\xi}^{\alpha}F(t,s,\xi)|\leq N|\xi|^{\gamma-|\alpha|}\exp\{-\kappa|\xi|^{\gamma}\},\quad\forall\xi\in\mathbf{R}^{d}\setminus\{0\}
$$
for every multi-index $\alpha$ with $|\alpha|\leq\lfloor\frac{d}{2}\rfloor+1$.
Therefore, by Lemma \ref{frac est},
$$
\|(-\Delta_{\xi})^{\sigma/2}F(t,s,\xi)\|_{L_{2}(\mathbf{R}^{d})}\leq N(d,\kappa,\gamma)
$$
 for all $\sigma\in[0,\frac{d}{2}+\gamma)\cap[0,\lfloor\frac{d}{2}\rfloor+1]$.

We claim that for any $\mu\in[0,\min\{\gamma,\lfloor\frac{d}{2}\rfloor+1-\frac{d}{2}\})$,
\begin{equation}
             \label{eq:integrability}
\int_{\mathbf{R}^{d}}|x|^{\mu}|K(t,s,x)|dx\leq N(d,\gamma,\kappa,\mu)(t-s)^{\frac{\mu}{\gamma}-1}.
\end{equation}
Indeed, fix $\mu\in[0,\min\{\gamma,\lfloor\frac{d}{2}\rfloor+1-\frac{d}{2}\})$
and choose $\sigma>0$ such that $$\mu+\frac{d}{2}<\sigma<\min\{\gamma+\frac{d}{2},\lfloor\frac{d}{2}\rfloor+1\}.
$$
Then by H\"older inequality, Parseval's identity, and Lemma \ref{frac est},
\begin{align*}
\int_{\mathbf{R}^{d}} & |x|^{\mu}|K(t,s,x)|dx\\
 & =(t-s)^{\frac{d+\mu}{\gamma}}\int_{\mathbf{R}^{d}}|x|^{\mu}|K(t,s,(t-s)^{\frac{1}{\gamma}}x)|dx\\
 & \leq N(t-s)^{\frac{d+\mu}{\gamma}}\left(\int_{\mathbf{R}^{d}}|(1+|x|^{\sigma})|K(t,s,(t-s)^{\frac{1}{\gamma}}x)|^{2}dx\right)^{1/2}\\
 & \leq N(t-s)^{\frac{\mu}{\gamma}-1}\left(\int_{\mathbf{R}^{d}}\left|(1+(-\Delta_{\xi})^{\sigma/2})F(t,s,\xi)\right|^{2}d\xi\right)^{1/2} \leq N(t-s)^{\frac{\mu}{\gamma}-1}.
\end{align*}
Hence \eqref{eq:integrability} holds for any $0\leq\mu<\min\{\gamma,\lfloor\frac{d}{2}\rfloor+1-\frac{d}{2}\}$. One also can see that
\begin{align*}
\int_{s}^{b}\int_{|x|\geq\rho}|K(t,s,x)|dxdt & \leq\int_{s}^{b}\int_{\mathbf{R}^{d}}\frac{|x|^{\mu}}{\rho^{\mu}}|K(t,s,x)|dxdt\\
 & \leq N\int_{s}^{b}\rho^{-\mu}(t-s)^{\frac{\mu}{\gamma}-1}dt =N\rho^{-\mu}(b-s)^{\frac{\mu}{\gamma}}.
\end{align*}
Therefore, $K$ satisfies \eqref{as 1 3}
with $\varphi(t)=Nt^{\mu}$ for some constant $N=N(d,\gamma,\kappa)$.

Next we prove \eqref{as 1 1} and \eqref{as 1 2}. Note that
\begin{align*}
(t-s)^{1+\frac{d+1}{\gamma}} & \frac{\partial K}{\partial x}(t,s,(t-s)^{\frac{1}{\gamma}}x)\\
 & =N1_{s<t}(t-s)^{1+\frac{d+1}{\gamma}}\int_{\mathbf{R}^{d}}e^{i((t-s)^{\frac{1}{\gamma}}x,\xi)}i\xi|\xi|^{\gamma}\exp\left\{ \int_{s}^{t}\psi(r,\xi)dr\right\} d\xi\\
 & =N1_{s<t}(t-s)^{1+\frac{d+1}{\gamma}}\int_{\mathbf{R}^{d}}e^{i(x,(t-s)^{\frac{1}{\gamma}}\xi)}i\xi|\xi|^{\gamma}\exp\left\{ \int_{s}^{t}\psi (r,\xi)dr\right\} d\xi\\
 & =N1_{s<t}\int e^{i(x,\xi)}i\xi|\xi|^{\gamma}\exp\left\{ \int_{s}^{t}\psi (r,(t-s)^{-\frac{1}{\gamma}}\xi)dr\right\} d\xi.
\end{align*}
For $\frac{d}{2}<\sigma<\min\{\gamma+1+\frac{d}{2},\lfloor\frac{d}{2}\rfloor+1\}$,
by H\"older inequality, Parseval's identity, and Lemma \ref{frac est},
\begin{align*}
\int_{\mathbf{R}^{d}} & |K(t,s,x+y)-K(t,s,x)|dx\\
 & =|y|\int_{\mathbf{R}^{d}}|\nabla K(t,s,x+\theta y)|dx    \quad \quad \quad (\theta\in [0,1])\\
 & =|y|(t-s)^{\frac{d}{\gamma}}\int_{\mathbf{R}^{d}}|\nabla K(t,s,(t-s)^{\frac{1}{\gamma}}x)|dx\\
 & \leq|y|(t-s)^{\frac{d}{\gamma}}\left(\int_{\mathbf{R}^{d}}|(1+|x|^{\sigma})\nabla K(t,s,(t-s)^{\frac{1}{\gamma}}x)|^{2}dx\right)^{1/2}\\
 & \leq N|y|(t-s)^{-\frac{1}{\gamma}-1}\left(\int_{\mathbf{R}^{d}}\left|(1+(-\Delta_{\xi})^{\sigma/2})\xi F(t,s,\xi)\right|^{2}d\xi\right)^{1/2}\\
 & \leq N|y|(t-s)^{-\frac{1}{\gamma}-1}.
\end{align*}
Hence
$$
\int_{a}^{\infty}\int_{\mathbf{R}^{d}}|K(t,s,x+y)-K(t,s,x)|dxdt\leq N|y|(a-s)^{-\frac{1}{\gamma}},
$$
and therefore \eqref{as 1 1} holds. Finally, denote
$$
\bar{F}(t,s,\xi)=1_{s<t}\psi(s,(t-s)^{-\frac{1}{\gamma}}\xi)|\xi|^{2\gamma}\exp\left\{\int_{s}^{t}\psi(r,(t-s)^{-\frac{1}{\gamma}}\xi)dr\right\}
$$
and observe that for $s\leq\max\{r,s\}\leq b<a<t$ and $\frac{d}{2}<\sigma<\min\{2\gamma+\frac{d}{2},\lfloor\frac{d}{2}\rfloor+1\}$,
(write $\tau=\theta s+(1-\theta)r$)
\begin{align*}
\int_{\mathbf{R}^{d}} & |K(t,r,x)-K(t,s,x)|dx\\
 & =|s-r|\int_{\mathbf{R}^{d}}|\partial_{s}K(t,\tau,x)|dx\\
 & =|s-r|(t-\tau)^{\frac{d}{\gamma}}\int_{\mathbf{R}^{d}}|\partial_{s}K(t,\tau,(t-\tau)^{\frac{1}{\gamma}}x)|dx\\
 & =\frac{|s-r|}{(t-\tau)^{2}}\int_{\mathbf{R}^{d}}|\mathcal{F}^{-1}\left\{ \bar{F}(t,\tau,\cdot)\right\} (x)|dx\\
 & \leq N\frac{|s-r|}{(t-\tau)^{2}}\left(\int_{\mathbf{R}^{d}}\left|(1-\Delta_{\xi}^{\sigma/2})\bar{F}(t,\tau,\xi)\right|^{2}d\xi\right)^{1/2}\\
 & \leq N\frac{|s-r|}{(t-b)^{2}}1_{\tau<t}.
\end{align*}
Therefore,
\begin{align*}
&\int_{a}^{\infty}\int_{\mathbf{R}^{d}}|K(t,r,x)-K(t,s,x)|dxdt \\& \leq N\int_{a}^{\infty}\frac{|s-r|}{(t-b)^{2}}dt
  \leq N|s-r|\int_{a-r}^{\infty}\frac{1}{t^{2}}dt
  =\frac{N|s-r|}{a-r}.
\end{align*}
This certainly leads to \eqref{as 1 2}, and thus  Assumption \ref{as 1} holds.

 {\bf{Part 2}}. We prove  \eqref{Gf estimate} when $p=q$. First we show $K$ satisfies \eqref{p0 conti} with $p_0=2$.
Due to Parseval's ideneity, for any $f \in L_2(\fR^{d+1})$ it holds that
\begin{align*}
 & \|\mathcal{G}f\|_{L_{2}(\mathbf{R}^{d+1})}^{2}\\
 & =N\int_{-\infty}^{\infty}\int_{\mathbf{R}^{d}}\left|\int_{-\infty}^{t}|\xi|^{\gamma}\exp\left\{ \int_{s}^{t}\psi(r,\xi)dr\right\} \mathcal{F}(f)(s,\xi)ds\right|^{2}d\xi dt\\
 & \leq N\int_{-\infty}^{\infty}\int_{\mathbf{R}^{d}}\left|\int_{-\infty}^{t}|\xi|^{\gamma}\exp\left\{ -\kappa|\xi|^{\gamma}(t-s)\right\} |\mathcal{F}(f)(s,\xi)|ds\right|^{2}d\xi dt\\
 & =N\int_{-\infty}^{\infty}\int_{\mathbf{R}^{d}}\left|\int_{-\infty}^{\infty}e^{it\tau}\left[\int_{-\infty}^{\infty} 1_{s<t}|\xi|^{\gamma}\exp\left\{ -\kappa|\xi|^{\gamma}(t-s)\right\} |\mathcal{F}(f)|(s,\xi)ds\right]dt\right|^{2}d\xi d\tau\\
 & =N\int_{-\infty}^{\infty}\int_{\mathbf{R}^{d}}\left|\left(\int_{0}^{\infty}e^{it\tau}|\xi|^{\gamma}\exp\{-\kappa|\xi|^{\gamma}t\}dt\right)\left(\int_{-\infty}^{\infty}e^{it\tau} |\mathcal{F}(f)|(t,\xi)dt\right)\right|^{2}d\xi d\tau\\
 & \leq N\int_{-\infty}^{\infty}\int_{\mathbf{R}^{d}}\left|\frac{|\xi|^{\gamma}}{i\tau-\kappa|\xi|^{\gamma}}\right|^{2}\left|\int_{-\infty}^{\infty}e^{it\tau} |\mathcal{F}(f)|(t,\xi)dt\right|^{2}d\xi d\tau\\
 & \leq N\int_{-\infty}^{\infty}\int_{\mathbf{R}^{d}}| \mathcal{F}(f)|^{2}d\xi dt=N\|f\|_{L_{2}(\mathbf{R}^{d+1})}^{2}.
\end{align*}
Actually $\cG f$ was defined only for $f\in C^{\infty}_0(\fR^{d+1})$. However, the calculations above show it is also defined on $L_2(\fR^{d+1})$.
Therefore $K$ satisfies \eqref{p0 conti} with $p_0=2$.
 Hence by Theorem \ref{main 11},  (\ref{Gf estimate})  holds for $p=q$, $1<p \leq 2$, and  for all $f\in L_{p}(\mathbf{R}^{d+1})$. For $p \in(2,\infty)$, we apply the standard duality argument. Denote $p'=p/(p-1)$ and
$$
P(s,t,x)=K(-t,-s,x)=1_{t<s}\mathcal{F}^{-1}\left\{|\xi|^{\gamma}\exp\left(\int_{t}^{s}\psi(-r,\xi)dr\right)\right\},
$$
and define operator $\mathcal{P}:L_{p'}(\mathbf{R}^{d+1})\rightarrow L_{p'}(\mathbf{R}^{d+1})$ by
$$
\mathcal{P}g(s,y)=\int_{\mathbf{R}^{d+1}}P(s,t,y-x)g(t,x)dxdt.
$$
Note that $\psi(-r,\xi)$ also satisfies \eqref{ellipticity} and \eqref{differentiability}.
Then for $g\in C_{0}^{\infty}(\mathbf{R}^{d+1})$, by change of variable $(t,s,x,y)\rightarrow(-t,-s,-x, -y)$ and Fubini's theorem we have
\begin{align}\label{duality}
\int_{\mathbf{R}^{d+1}}&g(t,x)\mathcal{G}f(t,x)dxdt\nonumber
\\ &=\int_{\mathbf{R}^{d+1}}g(t,x)\left(\int_{\mathbf{R}^{d+1}}K(t,s,x-y)f(s,y)dyds\right)dxdt\nonumber
\\ &=\int_{\mathbf{R}^{d+1}}f(s,y)\left(\int_{\mathbf{R}^{d+1}}K(t,s,x-y)g(t,x)dxdt\right)dyds
\\ &=\int_{\mathbf{R}^{d+1}}f(-s,-y)\left(\int_{\mathbf{R}^{d+1}}P(s,t,y-x)g(-t,-x)dxdt\right)dyds\nonumber
\\ &=\int_{\mathbf{R}^{d+1}}f(-s,-y)\mathcal{P}\bar{g}(s,y)dyds,\nonumber
\end{align}
where $\bar{g}(t,x)=g(-t,-x)$. By applying H\"{o}lder inequality,
$$
\int_{\mathbf{R}^{d+1}}g(t,x)\mathcal{G}f(t,x)dxdt\leq N\|f\|_{L_{p}(\mathbf{R}^{d+1})}\|\mathcal{P}\bar{g}\|_{L_{p'}(\mathbf{R}^{d+1})}.
$$
 Since $p' \in (1,2]$, we have $\|\mathcal{P}\bar{g}\|_{L_{p'}(\mathbf{R}^{d+1})}\leq N\|g\|_{L_{p'}(\mathbf{R}^{d+1})}$.
  This implies the desired result since $g\in C_{0}^{\infty}(\mathbf{R}^{d+1})$ is arbitrary. Thus \eqref{Gf estimate} holds for all $p\in (1,\infty)$.

{\bf{Part 3}}. Finally we check that $K$ satisfies \eqref{pq con} and prove \eqref{Gf estimate} for general $p,q>1$. Recall the operator $\mathcal{K}(t,s)$, that is
$$
\mathcal{K}(t,s)f(x)=\int_{\mathbf{R}^{d}}K(t,s,x-y)f(y)dy
$$
for $f\in C_0^\infty$ and $t,s\in\mathbf{R}$. Fix $p\in (1,\infty)$. By \eqref{eq:integrability}, we have
\begin{align*}
\|\mathcal{K}(t,s)f\|_{L_{p}}\leq N\|f\|_{L_{p}}\int_{\mathbf{R}^{d}}|K(t,s,y)|dy \leq N \|f\|_{L_{p}}(t-s)^{-1}.
\end{align*}
Hence \eqref{pq con} is satisfied with $\varphi(t)= t^{-1}$ and $p_{0}=p$.
Therefore from Theorem \ref{main 22} we conclude that for any $1<q\leq p$, \eqref{Gf estimate} holds for all $f\in C_{0}^{\infty}(\mathbf{R},L_{p})$.

Now let $1<p<q<\infty$. Define $p'=p/(p-1)$, $q'=q/(q-1)$. Since $1<q'<p'$, by \eqref{duality} we conclude that
\begin{align*}
\int_{\mathbf{R}^{d+1}}g(t,x)\mathcal{G}f(t,x)dxdt &=\int_{\mathbf{R}}\left(\int_{\mathbf{R}^{d}}f(-s,-y)\mathcal{P}\bar{g}(s,y)dy\right)ds \\
&\leq \int_{\mathbf{R}}\|f(-s,\cdot)\|_{L_{p}}\|\mathcal{P}\bar{g}(s)\|_{L_{p'}}ds\nonumber\\
&\leq N\|f\|_{L_{q}(\mathbf{R},L_{p})}\|g\|_{L_{q'}(\mathbf{R},L_{p'})}
\end{align*}
for any $f,g\in C_{0}^{\infty}(\fR^{d+1})$. Since $g$ is arbitrary and
$$
\|\mathcal{G}f\|_{L_{q}(\mathbf{R},L_{p})}=\sup_{\|g\|_{L_{q'}(\mathbf{R},L_{p'})}\leq 1}\left|\int_{\mathbf{R}^{d+1}}g(t,x)\mathcal{G}f(t,x)dxdt\right|,
$$
we have
$$
\|\mathcal{G}f\|_{L_{q}(\mathbf{R},L_{p})} \leq N \|f\|_{L_{q}(\mathbf{R},L_{p})}.
$$
Therefore for any $p, q \in (1, \infty)$, we obtain
\begin{align}
                    \label{lp lq est}
\|\cG f \|_{L_q(\fR,L_{p})}  \leq N \| f \|_{L_q(\fR,L_{p})}, \quad \forall f \in C_0^\infty(\fR^{d+1}),
\end{align}
where $N$ is independent of $f$. Since $C_0^\infty(\fR^{d+1})$ is dense in $L_q(\fR, L_p)$, $\cG$ is continuously extendible to $L_q(\fR, L_{p})$.
\end{proof}

Next, we prove a priori estimate.

\begin{lemma}[a priori estimate]
			\label{a priori lemma}
Let $\lambda \geq 0$ and $p,q \in (1,\infty)$. Suppose that conditions \eqref{ellipticity} and \eqref{differentiability} are fulfilled.
Then for any $u\in C_{0}^{\infty}(\mathbf{R}^{d+1})$, we
have
\begin{align}
&\|(-\Delta)^{\gamma/2}u\|_{L_q(\fR, L_p)}+\lambda\|u\|_{L_q(\fR, L_p)} \leq N\|u_t-\mathcal{A}u+\lambda u\|_{L_q(\fR, L_p)},
                \label{a priori est}
\end{align}
where $N$ depends only on $d,p,\gamma,$ and $\kappa$.
\end{lemma}
\begin{proof}
Put $f:=\frac{\partial}{\partial t}u-\mathcal{A}u+\lambda u$. Then obviously $f \in L_2(\fR^{d+1})$.

{\bf Case 1} $\lambda =0$. By taking the Fourier transform, we can easily check that
$$
(-\Delta)^{\gamma/2} u = \cG f \quad (a.e).
$$
Hence from \eqref{lp lq est}, we have
\begin{align}
                \label{lambda zero}
\|(-\Delta)^{\gamma/2}u\|_{L_q(\fR, L_p)}\leq N\|u_t-\mathcal{A}u\|_{L_q(\fR, L_p)}.
\end{align}

{\bf Case 2} $\lambda >0$. Similarly one can also check $ u = \cR f ~ (a.e)$.
Hence \eqref{a priori est} is a consequence of \eqref{norm of R} and \eqref{lambda zero} because
$$
\|u_t-\mathcal{A}u\|_{L_q(\fR, L_p)} \leq \|f\|_{L_q(\fR, L_p)} + \lambda \|u\|_{L_q(\fR, L_p)}.
$$
The lemma is proved.
\end{proof}

\textbf{Proof of Theorem \ref{main appl}}

Note that $C_0^\infty (\fR^{d+1})$ is dense in $\bH_{q,p}^{1,\gamma}$  and
$\cA(t)$ is a continuous operator on $H_{p}^{\gamma}$ due to Mihlin multiplier theorem.
Indeed, for $|\alpha|\leq \lfloor\frac{d}{2}\rfloor+1$,
$$
\left|D_{\xi}^{\alpha}\left(\frac{\psi(t,\xi)}{|\xi|^{\gamma}}\right)\right|\leq N\frac{|\xi|^{2|\alpha|\gamma-|\alpha|}}{|\xi|^{2\gamma|\alpha|}}\leq N(\kappa)|\xi|^{-|\alpha|}.
$$
Hence from Lemma \ref{a priori lemma},
\begin{align*}
\notag &\|u_{t}\|_{L_q(\fR, L_p)}+\|(-\Delta)^{\gamma/2}u\|_{L_q(\fR, L_p)}+\lambda\|u\|_{L_q(\fR, L_p)} \\
&\leq N(d,p,q,\gamma,\kappa)\|u_t-\mathcal{A}u+\lambda u\|_{L_q(\fR, L_p)}
\end{align*}
 for any $u \in \bH_{q,p}^{1,\gamma}$, and the uniqueness of  solutions to  \eqref{eqn 8.13} is proved.

It only remains to prove the existence of solutions. For $f \in \bH_{q,p}^{1,\gamma}$, we consider a sequence $f_n \in C_0^\infty(\fR^{d+1})$ so that
$ \|f_n -f\|_{\bH_{q,p}^{1,\gamma}} \to 0$ as $n \to \infty$. For each $n$, we can easily check that $\cR f_n$ is a solution to  \eqref{eqn 8.13}.
Since $\bH_{q,p}^{1,\gamma}$ is a Banach space, we can find a solution $u$ as the limit of $\cR f_n$ in $\bH_{q,p}^{1,\gamma}$ using the a priori estimate.
The theorem is proved. \qed

\bibliographystyle{amsplain}

\begin{thebibliography}{10}

\bibitem{caffarelli1998w}
Luis~A Caffarelli and I~Peral, \emph{On {$W^{1}_{p}$} estimates for elliptic
  equations in divergence form}, Communications on pure and applied mathematics
  \textbf{51} (1998), no.~1, 1--21.

\bibitem{fabes1966singular}
E~Fabes and N~Riviere, \emph{Singular integrals with mixed homogeneity}, Studia
  Mathematica \textbf{27} (1966), no.~1, 19--38.

\bibitem{gilbarg2001elliptic}
David Gilbarg and Neil~S Trudinger, \emph{Elliptic partial differential
  equations of second order}, vol. 224, springer, 2001.

\bibitem{grafakos2008classical}
Loukas Grafakos, \emph{Classical {Fourier} analysis}, vol. 249, Springer, 2008.

\bibitem{grafakos2009modern}
\bysame, \emph{Modern {Fourier} analysis: {Structure} of topological groups,
  integration theory, group representations. vol. 1}, vol. 250, Springer, 2009.

\bibitem{iwaniec1983projections}
Tadeusz Iwaniec, \emph{Projections onto gradient fields and {$L_{p}$}-estimates
  for degenerated elliptic operators}, Studia Mathematica \textbf{75} (1983),
  no.~3, 293--312.

\bibitem{jones1964class}
B~Frank Jones, \emph{A class of singular integrals}, American Journal of
  Mathematics (1964), 441--462.

\bibitem{Kim2014BMOpseudo}
Ildoo Kim, Kyeong-Hun Kim, and Sungbin Lim, \emph{Parabolic {BMO} estimates for
  pseudo-differential operators of arbitrary order}, arXiv preprint
  arXiv:1408.2343 (2014).

\bibitem{krylov2001calderon}
Nicolai~V Krylov, \emph{On the {Calder{\'o}n-Zygmund} theorem with applications
  to parabolic equations}, Algebra i Analiz \textbf{13} (2001), no.~4, 1--25.

\bibitem{Krylov2002}
\bysame, \emph{The {Calder{\'o}n-Zygmund} theorem and parabolic equations in {$
  L_p (\mathbb{R}, C^{2+\alpha}) $}-spaces}, Annali della Scuola Normale
  Superiore di Pisa-Classe di Scienze \textbf{1} (2002), no.~4, 799--820.

\bibitem{lara2014regularity}
H{\'e}ctor~Chang Lara and Gonzalo D{\'a}vila, \emph{Regularity for solutions of
  non local parabolic equations}, Calculus of Variations and Partial
  Differential Equations \textbf{49} (2014), no.~1-2, 139--172.

\bibitem{Mikulevivcius1992}
R~Mikulevi{\v{c}}ius and H~Pragarauskas, \emph{On the {Cauchy} problem for
  certain integro-differential operators in {Sobolev and H{\"o}lder} spaces},
  Lithuanian Mathematical Journal \textbf{32} (1992), no.~2, 238--264.

\bibitem{zhang2013lp}
Xicheng Zhang, \emph{{$Lp$}-maximal regularity of nonlocal parabolic equations
  and applications}, Annales de l'Institut Henri Poincare (C) Non Linear
  Analysis, vol.~30, Elsevier, 2013, pp.~573--614.

\end{thebibliography}

\end{document}